\documentclass[smallextended]{svjour3}                     

\smartqed  
\usepackage{geometry}                
\usepackage{graphicx}
\usepackage{amssymb}
\usepackage{epstopdf}
\usepackage{pgf,tikz}
\usepackage{amsmath}
\usepackage{tikz}
\usepackage{graphicx}
\usepackage{mathtools}
\usepackage{graphicx}
\usepackage{epstopdf}
\usepackage{fix-cm}
\usepackage{geometry}                
\usepackage{graphicx}
\usepackage{amssymb}
\usepackage{epstopdf}
\usepackage{pgf,tikz}
\usepackage{latexsym}
\usepackage{amsmath,amsfonts}
\usepackage{graphicx}
\usepackage{geometry}
\usepackage{graphicx}
\usepackage{multirow}
\usepackage{amssymb}
\usepackage{epstopdf}
\usepackage{pgf,tikz}
\usepackage{latexsym}
\usepackage[english]{babel}
\usepackage{dcolumn}
\usepackage{color}
\newcommand{\ignore}[1]{}

\newcommand{\R}{{\mathbb R}}
\newcommand{\oN}{{\mathbb N}}

\newcommand{\Supp}{\text{\rm Supp}}

\journalname{}
\begin{document}

\title{On the convergence rate of grid search for polynomial optimization over the  simplex}


\author{Etienne de Klerk         \and
        Monique Laurent          \and
        Zhao Sun\and
        Juan C. Vera}

\institute{Etienne de Klerk \at
              Tilburg University \at
              PO Box 90153, 5000 LE Tilburg, The Netherlands \\
              \email{E.deKlerk@uvt.nl}
           \and
           Monique Laurent \at
              Centrum Wiskunde \& Informatica (CWI), Amsterdam and Tilburg University \at
              CWI, Postbus 94079, 1090 GB Amsterdam, The Netherlands\\
              \email{M.Laurent@cwi.nl}
           \and
           Zhao Sun \at
              \'Ecole Polytechnique de Montr\'eal \at
              GERAD--HEC Montreal 3000, C\^ote-Sainte-Catherine Rd, Montreal, QC H3T 2A7, Canada \\
              \email{Zhao.Sun@polymtl.ca}
              \and
              Juan C. Vera\at
              Tilburg University \at
              PO Box 90153, 5000 LE Tilburg, The Netherlands \\
              \email{J.C.Veralizcano@uvt.nl }
}

\date{Received: date / Accepted: date}

\maketitle

\begin{abstract}
{
We consider the approximate minimization of a given polynomial on the standard simplex, obtained by taking
the minimum value over all rational grid points with given denominator ${r} \in \mathbb{N}$.
It was shown in
[De Klerk, E., Laurent, M., Sun, Z.:
An error analysis for polynomial optimization over the simplex based on the multivariate hypergeometric distribution.
{\em SIAM J. Optim.} 25(3) 1498--1514 (2015)]
that the 
 accuracy of this approximation depends on $r$ as $O(1/r^2)$ if
there exists a rational global minimizer. In this note we show that the rational minimizer condition is not necessary to obtain the
$O(1/r^2)$ bound.
}

{
}
\end{abstract}

\section{Introduction}

\noindent
We consider the problem of minimizing a  polynomial $f$ over the standard simplex $$\Delta_n=\left\{x\in\R^n_+:\sum_{i=1}^n x_i= 1\right\}.$$
That is, the problem of finding
\begin{equation}
\label{problem}
f_{\min,\Delta_n}=\min f(x)\ \ \text{s.t.}\ \ x\in\Delta_n.
\end{equation}
\noindent
Analogously, we denote $\displaystyle f_{\max,\Delta_n}=\max_{x\in\Delta_n} f(x).$

\smallskip\noindent
We consider the parameter $f_{\min,\Delta(n,r)}$ obtained by minimizing $f$ over the regular grid $\Delta(n,r)=\left\{x\in\Delta_n:rx\in\oN^n\right\}$,
consisting of all rational points in $\Delta_n$ with denominator $r$. That is,
\begin{eqnarray*}
f_{\min,\Delta(n,r)}=\min f(x)\ \ \text{s.t.}\ \ x\in\Delta(n,r).
\end{eqnarray*}
Note that the calculation of $f_{\min,\Delta(n,r)}$ requires $\left|\Delta(n,r)\right| = {n+r-1 \choose r}$ function evaluations. Thus it may be computed in
polynomial time for fixed $r$.

Interestingly, the parameter $f_{\min,\Delta(n,r)}$  yields
a polynomial-time approximation scheme (PTAS) for problem (\ref{problem}) for polynomials of fixed degree, in the sense of the following two theorems.

The first theorem deals with the quadratic function case, and is due to Bomze and De Klerk \cite{BK02}.
\begin{theorem}   \cite[Theorem 3.2]{BK02} 
\label{thmklsquad0}
For any quadratic polynomial $f$ and $r\ge 1$, one has
 $$f_{\min,\Delta(n,r)}-f_{\min,\Delta_n}\le {f_{\max,\Delta_n} - f_{\min,\Delta_n}\over r}.$$
\end{theorem}
One says that $f_{\min,\Delta(n,r)}$ approximates $f_{\min,\Delta_n}$ with relative accuracy $1/r$, where the relative accuracy is defined
as the ratio $\left(f_{\min,\Delta(n,r)}-f_{\min,\Delta_n}\right)/\left(f_{\max,\Delta_n} - f_{\min,\Delta_n}\right)$.
Note that this definition of a PTAS
 is that one may approximate $f_{\min,\Delta_n}$ to within any fixed relative accuracy in polynomial
time. (This definition was introduced in the late 1970s, see e.g.\ \cite{AAP,BR} and the references therein.)
In particular, for any fixed $\epsilon > 0$, one has relative accuracy at most $\epsilon$ for $r \ge 1/\epsilon$.
(Recall that $f_{\min,\Delta(n,r)}$ may be computed in polynomial time for fixed $r$.)

The second theorem is an extension of the previous result to polynomial objectives of fixed degree, and is due to
De Klerk,   Laurent and Parrilo \cite{KLP06}.

\begin{theorem}\cite[Theorem 1.3]{KLP06}\label{thmklsgeneral}
For any  polynomial $f$ of degree $d$ and $r\ge 1$,  one has
\begin{equation*}
\label{final bounds}
f_{\min,\Delta(n,r)}-f_{\min,\Delta_n} \le 
  \left(1-{r^{\underline{d}}\over r^d}\right){2d-1\choose d}d^d\left(f_{\max,\Delta_n} - f_{\min,\Delta_n}\right)
  \le {C_d\over r} \left(f_{\max,\Delta_n} - f_{\min,\Delta_n}\right),
  \end{equation*}
where $r^{\underline{d}}:=r(r-1)\cdots(r-d+1)$ denotes the {\em falling factorial} and $C_d$ is a constant depending only on $d$.
\end{theorem}

Once again, one has that $f_{\min,\Delta(n,r)}$ approximates $f_{\min,\Delta_n}$ with relative accuracy $O(1/r)$, if $d$ is fixed.
(Here the constant in the big-O notation depends on $d$ only, i.e., for fixed $d$ it is an absolute constant not depending on the polynomial $f$.)

The authors of \cite{KLS15} show that there does not exist an $\epsilon > 0$ and a constant $C>0$ such that, for any quadratic form $f$,
$$f_{\min,\Delta(n,r)}-f_{\min,\Delta_n}\le { C \over r^{1+\epsilon}}(f_{\max,\Delta_n} - f_{\min,\Delta_n})
 \quad \forall r \in \mathbb{N},$$
so in this sense the $1/r$  bound on the relative accuracy is tight in Theorem \ref{thmklsquad0}.

On the other hand if, as opposed to the PTAS property, one is only interested in the dependence of the 
accuracy $ f_{\min,\Delta(n,r)}-f_{\min,\Delta_n}$ on $r$, 
then one may obtain $O(1/r^2)$ bounds, as shown in \cite{KLSSIAM}.
Here the constant in the big-O notation may depend on the polynomial $f$.
For example, for  a 
quadratic polynomial $f$, De Klerk et al. \cite{KLSSIAM} show the following result.

\begin{theorem}\label{thmklsquad}\cite[Theorem 2.2]{KLSSIAM}
Let $f$ be a 
quadratic polynomial, and let $x^*$ be a global minimizer of $f$ over $\Delta_n$, with denominator $m$, i.e.\ $mx^* \in \mathbb{N}^n$.
For all integers $r\ge 1$,  one has
\begin{equation*}
f_{\min,\Delta(n,r)} - f_{\min,\Delta_n} \le  \frac{m}{r^2}\left(f_{\max,\Delta_n} - f_{\min,\Delta_n}\right).
\end{equation*}
\end{theorem}
Note that this result does not give a PTAS, since the relative error is $m/r^2$. (To get a given relative accuracy $\epsilon >0$,
one needs $r \ge \sqrt{m/\epsilon}$, so that $r$ then depends on the problem size.)

The proof of \cite[Theorem 2.2]{KLSSIAM} relied on the fact that
for quadratic objective functions  the {problem} (\ref{problem})
has a rational global minimizer.
For higher degree objective functions, the authors of \cite{KLSSIAM} could
only prove the $O(1/r^2)$ bound under the (restrictive) assumption of the existence of a rational minimizer.

\begin{theorem}\label{thmklsgene}\cite[Corollary 4.5 and Lemma 4.6]{KLSSIAM}
Let $f$ be a polynomial of degree $d$ and assume that $f$ has a rational global minimizer over $\Delta_n$ (say, in $\Delta(n,m)$). Then, one has
\begin{equation*}
f_{\min,\Delta(n,r)} - f_{\min,\Delta_n} \le \frac{mc_d}{r^2}(f_{\max,\Delta_n}-f_{\min,\Delta_n}),
\end{equation*}
for some constant $c_d$ depending only on $d$, namely\footnote{This value of $c_d$ can be easily derived from results in \cite{KLSSIAM} (specifically from Theorem 4.1, Lemma 4.6 and its proof).}
$c_d=(d-1)(d!-1)d^{2d-1}{2d-1\choose d}$.
\end{theorem}

In this note we prove that the accuracy $f_{\min,\Delta(n,r)} - f_{\min,\Delta_n}$ is $O(1/r^2)$  without the rational minimizer assumption.
More precisely we show that for any polynomial $f$ there exists a constant $C_f$ (depending on $f$) such that
$f_{\min,\Delta(n,r)} - f_{\min,\Delta_n}\le {C_f\over r^2}$ for all $r\in\oN$. We will give several bounds, involving different constants $C_f$.
For the first bound  in Theorem \ref{corst}, the constant $C_f$ depends on the support of a global minimizer of $f$ and the coefficients of $f$ 
while, for the second bound in Theorem \ref{thm2},
it depends on the smallest positive component of the minimizer, the range of values $f_{\max,\Delta_n}-f_{\min,\Delta_n}$  and the degree of $f$.

The results in this note complement a growing literature on the complexity of
polynomial optimization and interpolation on a simplex; see \cite{BK02,BGY,Bos83,KLP06,KLS15,KLSSIAM,Fay03,Nes2003}  and the references therein.

\subsection*{Notation}
For an integer $n\ge 1$, we let $[n]=\{1,2\dots,n\}$.
We denote $\oN^n_d=\left\{\alpha\in\oN^n:\sum_{i=1}^n\alpha_i\le d\right\},$ with $\mathbb{N}$ the set of nonnegative integers.
For $x\in \R^n$  and  $\alpha\in \oN^n$, we set $x^{\alpha}=\prod_{i=1}^nx_i^{\alpha_i}$. Moreover, given a subset  $I\subseteq[n]$, $x_{I}$ denotes the vector in $\R^{|I|}$ that contains the components $x_i$ with $i\in I$.
Finally, the support of $x\in \R^n$ is the set $\Supp(x)=\{i\in [n]: x_i\ne 0\}.$

\section{Preliminary results}

First we will show some auxiliary results  about approximations by grid points.

\begin{lemma} \label{lemxr}
Let $x^*\in \Delta_n$ with support
$I=\{i\in [n]: x_i^* >0\}$. Then, for each integer $r\ge 1$, there exists a  point $\tilde x\in \Delta(n,r)$ such that
\begin{equation}\label{eqxr}
\|x^*-\tilde x\|_\infty \le {1\over r}\left(1-{1\over |I|}\right) \ \text{  and } \ \tilde x_i=0\ \ \forall i\in[n]\backslash I.
\end{equation}
\end{lemma}

\begin{proof}
First, we set $\tilde{x}_i=0$ at the positions $i\in[n]\backslash I$. Then, we define the values of $\tilde x_i$ with $i\in I$. By  \cite[Theorem 7]{BGY}, there exists a grid point $x'\in \Delta(|I|,r)$ such that
$\|x_I^*- x'\|_\infty \le {1\over r}\left(1-{1\over |I|}\right)$.
Set $\tilde x_I=x'$ and we get a point $\tilde x\in \Delta(n,r)$ satisfying \eqref{eqxr}.
\ignore{
Let $x^*\in \Delta_n$ and $r\ge 1$. Then, by  \cite[Theorem 7]{BGY}, there exists a grid point $x\in \Delta(n,r)$ such that
$\|x^*- x\|_\infty \le {1\over r}\left(1-{1\over n}\right)$.
Set $I=\{i\in [n]: x^*_i>0\}$ and $J=[n]\setminus I=\{j\in[n]:x^*_j=0\}$.
If $x_j=0$ for all $j\in J$, we set $\tilde x=x$ and we are done.
Otherwise, consider an index $j\in J$ for which $x_j>0$. We first claim that there exists an index $i\in I$ such that $x_i\le x^*_i$.
For, if not, then we would have that $1\ge \sum_{i\in I}x_i >\sum_{i\in I}x^*_i=1$, a contradiction.
We now consider the following point $x'\in \mathbb R^n$, with entries $x'_i=x_i+x_j$, $x'_j=0$ and $x'_k=x_k$ for all other indices $k\ne i,j$.
Then, $x'\in \Delta(n,r)$ (this is clear) and moreover $\|x'-x^*\|_\infty \le \|x^*-x\|_\infty$.
To see this it suffices to check that $|x_i+x_j-x^*_i|\le \|x^*-x\|_\infty$.
Indeed, either $x_i+x_j-x^*_i\ge 0$ in which case we have: $|x_i+x_j-x^*_i|= x_i+x_j-x^*_i\le x_j=x_j-x^*_j\le \|x^*-x\|_\infty$, or
$x_i+x_j-x^*_i\le 0$ in which case we have:
$|x_i+x_j-x^*_i|= x^*_i-x_i-x_j \le x^*_i-x_i \le \|x^*-x\|_\infty$.
After iterating this procedure with all coordinates indexed by $J$, we get a point $\tilde x\in \Delta(n,r)$ satisfying the lemma.}
\qed\end{proof}

\begin{lemma}\label{lem:lin}
Let $x^*$ be a global minimizer of the polynomial $f$ in $\Delta_n$ and let $\tilde x$ be a  point in $\Delta(n,r)$ satisfying (\ref{eqxr}). Then, one has
\[\nabla f(x^*)^T(\tilde x - x^*) = 0.\]
\end{lemma}

\begin{proof}
By assumption, $x^*$ is an optimal solution of the  optimization problem $\min\{f(x): x\ge 0, e^Tx = 1\}$. From the KKT (necessary) conditions (see, e.g., \cite[Chapter 5.5.3]{BV04}), we have that there exist $\mu \in \R$ and $\lambda \in \R^n_+$ such that
$\nabla f(x^*) = -\mu e  + \lambda $, and $\lambda_i x^*_i = 0$ for all  $i\in[n]$. Then we have
$\nabla f(x^*)^T(\tilde x - x^*) = -\mu e^T(\tilde x - x^*)  + \lambda^T(\tilde x - x^*)$. Moreover,  $e^T(\tilde x - x^*) = e^T\tilde x - e^T x^* = 1 -1 = 0$ and $\lambda _i >0 $ implies $x^*_i =0$ and thus $\tilde x_i =0$, so that $ \lambda^T\tilde x = 0 = \lambda^T x^*$. This shows $\nabla f(x^*)^T(\tilde x - x^*) = 0$.
\qed
\end{proof}

\begin{lemma}\label{lemhes}
Consider a polynomial $f=\sum_{\alpha\in\oN^n_d}f_{\alpha}x^{\alpha}$ of degree $d$. Then, for any point $x\in[0,1]^n$, one has
\[\sum_{i,j\in[n]}\left| \nabla^2 f(x)_{i,j} \right|\le d(d-1)\sum_{\alpha\in\oN^n_d}|f_{\alpha}|.\]
\end{lemma}

\begin{proof}
\smallskip\noindent
As $f(x)=\sum_{\alpha\in\oN^n_d}f_{\alpha}x^{\alpha}$, one has
\begin{displaymath}
\nabla^2 f(x)_{i,j} = \left\{ \begin{array}{ll}
\sum_{\alpha\in\oN^n_d}f_{\alpha}\alpha_i\alpha_j x^{\alpha-e_i-e_j} & \textrm{for $i\ne j$,}\\
\sum_{\alpha\in\oN^n_d}f_{\alpha}\alpha_i(\alpha_i-1) x^{\alpha-2e_i} & \textrm{for $i=j$.}
\end{array} \right.
\end{displaymath}

\smallskip\noindent
Thus, we have
\begin{align*}
\sum_{i,j\in[n]}|\nabla^2 f(x)_{i,j}|
&\le  \sum_{i,j\in[n]:i\ne j}\sum_{\alpha\in\oN^n_d}|f_{\alpha}|\alpha_i\alpha_jx^{\alpha-e_i-e_j}+\sum_{i=1}^n\sum_{\alpha\in\oN^n_d}|f_{\alpha}|\alpha_i(\alpha_i -1)x^{\alpha-2e_i}\\
&\le  \sum_{i,j\in[n]:i\ne j}\sum_{\alpha\in\oN^n_d}|f_{\alpha}|\alpha_i\alpha_j+\sum_{i=1}^n\sum_{\alpha\in\oN^n_d}|f_{\alpha}|\alpha_i(\alpha_i -1)\\
&= \sum_{i,j\in[n]}\sum_{\alpha\in\oN^n_d}|f_{\alpha}|\alpha_i\alpha_j -\sum_{i=1}^n\sum_{\alpha\in\oN^n_d}|f_{\alpha}|\alpha_i\\
&= \sum_{\alpha\in\oN^n_d}|f_{\alpha}|\left(\left(\sum_{i=1}^n \alpha_i\right)^2 - \left(\sum_{i=1}^n \alpha_i\right)\right) \\
&\le (d^2-d)\sum_{\alpha\in\oN^n_d}|f_{\alpha}|,
\end{align*}
where for the second inequality we use  $x_i\in [0,1]$ for any $i\in[n]$.
\qed
\end{proof}

\section{Bounds in terms of the support of a global minimizer}

In this section we prove the following result, which shows the $O(1/r^2)$ convergence  for  the upper bounds $f_{\min,\Delta(n,r)}$ without the restrictive assumption
of a rational minimizer.

\begin{theorem}\label{corst}
Consider a polynomial $f=\sum_{\alpha\in\oN^n_d}f_{\alpha}x^{\alpha}$ of degree $d$.
Let $x^*$  be a global minimizer of $f$ in $\Delta_n$ with support
$I=\{i\in [n]: x_i^* >0\}$.
Then, for all integers $r\ge 1$, one has
\begin{equation*}
f_{\min,\Delta(n,r)}- f_{\min,\Delta_n} \le {d(d-1)\over 2r^2}\left(1-{1\over |I|}\right)^2\sum_{\alpha\in\oN^n_d}|f_{\alpha}|.
\end{equation*}
\end{theorem}
\begin{proof}
Let $x^*\in \Delta_n$ be a global minimizer of $f$ in $\Delta_n$,
let $\tilde x\in \Delta(n,r)$ satisfying the condition (\ref{eqxr}) from Lemma \ref{lemxr}, and set $h=\tilde x-x^*$.
Using Taylor's theorem, we can write:
\begin{equation}\label{eqTaylor}
f(\tilde x) -f(x^*)= f(x^*+h)- f(x^*)= \nabla f(x^*)^Th+{1\over 2} h^T\nabla^2 f(\zeta) h,
\end{equation}
for some point $\zeta$ lying in the segment $[x^*,x^*+h]=[x^*,\tilde x]\subseteq \Delta_n$.
By Lemma \ref{lem:lin}, we know that $\nabla f(x^*)^Th=0$. Using (\ref{eqxr}) and Lemma \ref{lemhes}, we can upper bound the second term  as follows:
$${1\over 2}h^T\nabla^2 f(\zeta) h \le {1\over 2}\|h\|^2_\infty \sum_{i,j=1}^n |\nabla^2f(\zeta)_{i,j}| \le {d(d-1) \over 2r^2} \left(1-{1\over |I|}\right)^2 \sum_{\alpha\in\oN^n_d}|f_\alpha|.$$
Combining with  $f_{\min,\Delta(n,r)}-f_{\min,\Delta_n} \le f(\tilde x)-f(x^*)$, this concludes the proof.
\qed\end{proof}

\medskip
Note that when the support $I$ of the global minimizer $x^*$ is a singleton (i.e., $x^*$ is a standard unit vector), $f_{\min,\Delta_n}=f_{\min,\Delta(n,r)}$ for any $r\ge 1$, which is consistent with the inequality in Theorem~\ref{corst} (whose right hand side is equal to zero).

%
%

Note also that one can tighten the result of Theorem \ref{corst} by replacing the sum $\sum_{\alpha\in\oN^n_d}|f_{\alpha}|$ by
$\sum_{\alpha\in\oN^n_d: Supp(\alpha)\subseteq I}|f_{\alpha}|.$
For this, it suffices to apply Theorem \ref{corst} to the polynomial
$g(x_1,\ldots, x_{|I|})=f(x_1,\ldots,x_{|I|},0,\ldots,0)= \sum_{\alpha \in \mathbb N^n_d: Supp(\alpha)\subseteq I} f_\alpha x^\alpha$,  after observing that $x^*_I$ is a global minimizer of $g$ over the simplex $\Delta_{|I|}$ and that
$f_{\min,\Delta(n,r)}\le g_{\min, \Delta(|I|,r)}$ and
$f_{\min,\Delta_n}=g_{\min,\Delta_{|I|}}$.

We mention   another  variation of the bound in Theorem \ref{corst}, where the  quantity  $\sum_\alpha |f_\alpha|$ is now replaced by $\sum_\alpha | g_\alpha|$ for an appropriate polynomial $ g$ (depending on the support of a global minimizer of $f$).



\begin{corollary}\label{corlast}
Consider a polynomial $f$ of degree $d$.
 Let $x^*$ be a global minimizer of $f$ in $\Delta_n$ with support  $I=\{i\in [n]:x^*_i>0\}$, assumed to be equal to $\{1,\ldots,|I|\}$.
Define the $(|I|-1)$-variate polynomial  ${g}(x_1,\dots,x_{|I|-1})=f(x_1,\dots,x_{|I|-1},1-\sum_{i=1}^{|I|-1}x_i,0,\ldots,0)$ (with 0 at the positions $i\not\in I$).
For all integers $r\ge 1$,
one has
\begin{eqnarray*}
f_{\min,\Delta(n,r)}-f_{\min,\Delta_n}\le {d(d-1)\over 2r^2}\left(1-{1\over |I|}\right)^2\sum_{\alpha\in\oN^{|I|-1}_d}|{g}_{\alpha}|.
\end{eqnarray*}
\end{corollary}


\section{Bounds in terms of the smallest positive component of a global minimizer}

\noindent
We now give a different approach for the convergence rate of the bounds $f_{\min,\Delta(n,r)}$.
We will use the following well-known Euler's identity for homogeneous
 polynomials.

\begin{theorem}[Euler's Identity]\label{thm:Euler}
Let $f$ be an $n$-variate homogeneous  polynomial of degree $d$. Then, for all $k\leq d$,
 \[\sum_{i_i,\dots,i_k \in [n]}\frac{\partial^k f(x)}{\partial x_{i_1}\dots \partial x_{i_k}} x_{i_1}\dots x_{i_k} = \frac{d!}{(d-k)!}f(x).
 \]
\end{theorem}

We start with several preliminary results that we will need for our main result in Theorem~\ref{thm2} below.

\begin{lemma}\label{lem:non-neg-hom}
Consider a  homogeneous polynomial $f$ of degree $d \ge 1$, assumed to have nonnegative coefficients.
Let  $x^*$ be a global minimizer of $f$ on $\Delta_n$ and let $\tilde x\in \Delta(n,r)$ satisfying (\ref{eqxr}).
Consider a scalar  $s >0$  such that $|\tilde x_i - x^*_i| \le  s x^*_i$ for all $i \in [n]$.  Then, for all integers $r\ge 1$,
\begin{equation*}
f_{\min,\Delta(n,r)} - f_{\min,\Delta_n} \le ((1+s)^d - (1+ds))f_{\min,\Delta_n}.
\end{equation*}
\end{lemma}

\begin{proof} First note that, as $f$ has nonnegative coefficients then, for all $k\ge 1$,  $i_1,\dots,i_k \in [n]$ and  $x \in \Delta_n$, we have
\begin{equation}\label{eq:non-negPart}
\frac{\partial^k f(x)}{\partial x_{i_1}\dots \partial x_{i_k}} \ge 0.
\end{equation}

\noindent Set $h = \tilde x - x^*$. Then, we have:
 \begin{align*}
 f_{\min,\Delta(n,r)} - f_{\min,\Delta_n}
 &\le  f(\tilde x) - f(x^*) \\
 & = \sum_{k=1}^d \frac{1}{k!}\sum_{i_i,\dots,i_k \in [n]}\frac{\partial^k f(x^*)}{\partial x_{i_1}\dots \partial x_{i_k}} h_{i_1}\dots h_{i_k}
 & \text{(From Taylor's theorem)} \\
 & = \sum_{k=2}^d \frac{1}{k!}\sum_{i_1,\dots,i_k \in [n]}\frac{\partial^k f(x^*)}{\partial x_{i_1}\dots \partial x_{i_k}} h_{i_1}\dots h_{i_k}
 & \text{(from Lemma \ref{lem:lin})} \\
 & \le \sum_{k=2}^d \frac{s^k}{k!}\sum_{i_1,\dots,i_k \in [n]}\frac{\partial^k f(x^*)}{\partial x_{i_1}\dots \partial x_{i_k}} x^*_{i_1}\dots x^*_{i_k}
 & \text{(using \eqref{eq:non-negPart})} \\
 & = \sum_{k=2}^d s^k{d \choose k} f(x^*)
 & \text{(from Theorem \ref{thm:Euler})} \\
 & = ((1+s)^d - (1+ds)) f(x^*).
\end{align*}
\qed
\end{proof}

\begin{lemma}\label{lem:epsbound} Let $0 < \epsilon \le  2/3$.  For any scalar $s \ge 0$ such that  $ds  \le \epsilon$, we have
$$(1+s)^d - (1+ds) \le  (1+ \epsilon){d \choose 2}s^2.$$
\end{lemma}

\begin{proof} From the binomial theorem, we have:
$(1+s)^d - \left(1+ds + {d \choose 2}s^2\right)
 = \sum_{k=3}^d {d \choose k} s^k$.
 Hence it suffices to show that $\sum_{k=3}^d {d \choose k} s^k \le \epsilon {d \choose 2} s^2$.
 One can verify  that {${d\choose k} \le \frac {d^{k-2}}3 {d\choose 2}$ for all $k\ge 3$}. Using this and {$0 \le ds \le \epsilon\le 2/3$} one obtains:
 {\[ \sum_{k=3}^d {d \choose k} s^k \le \frac { s^2}3 {d\choose 2}  \sum_{k=3}^{d} (ds)^{k-2}\le  \frac { s^2}3 {d\choose 2}  \sum_{k=1}^{\infty} \epsilon^{k} =  \frac { s^2}3 {d\choose 2} \frac {\epsilon}{1-\epsilon} \le \epsilon{d\choose 2}s^2.
 \]}
\qed
\end{proof}

\begin{lemma}\label{lem:sbound}
Let $x^* \in \Delta_n$ be given. Let $r\ge 1$ and let $\tilde x\in \Delta(n,r)$ satisfying relation (\ref{eqxr}).
Let $x^*_{\min}$ be the smallest positive component of $x^*$. Then  $|\tilde x_i - x^*_i| \le  \frac{1}{rx^*_{\min}} x^*_i$ for all $i \in [n]$.
\end{lemma}

\begin{proof}
Fix $i \in [n]$. If $x^*_i = 0$, then $\tilde x_i = 0$ by (\ref{eqxr}) and thus the desired inequality holds. Otherwise,  $x^*_{\min} \le x^*_i$ and thus
$|\tilde x_i - x_i| \le  \frac{1}{r} \le \frac{1}{rx^*_{\min}} x^*_i$.
\qed
\end{proof}

\medskip
We can now state our main result of this section,
which shows again (but with a different constant) that the parameter $f_{\min,\Delta(n,r)}$ approximates $f_{\min,\Delta_n}$ with  accuracy in $O(1/r^2)$.

\begin{theorem}\label{thm2}
 Let $f$ be a  polynomial of degree $d \ge 1$. 
 Let $x^*$ be a global minimizer of $f$ on $\Delta_n$ with
 smallest positive component  $x^*_{\min}$.  Then, for any $0< \epsilon \le 2/3$, and $r\ge \frac{d}{\epsilon x^*_{\min}}$, one has
\begin{equation*}
f_{\min,\Delta(n,r)} - f_{\min,\Delta_n} \le {1\over r^2}  \frac{(1+\epsilon)d^d{d \choose 2}}{(x^*_{\min})^2}{2d-1\choose d}(f_{\max,\Delta_n}-f_{\min,\Delta_n}).
\end{equation*}
\end{theorem}

\begin{proof}
First observe that it suffices to show the result for  homogeneous polynomials.
Indeed, if $f=\sum_{\alpha}f_\alpha x^\alpha$ is not homogeneous, then we may consider instead the homogeneous polynomial
$F(x)=\sum_\alpha f_\alpha x^\alpha (\sum_{i=1}^nx_i)^{d-|\alpha|}$ and the result for $F$ will imply the result for $f$.
Hence we now assume that $f=\sum_{\alpha\in \oN_{=d}^n}f_{\alpha}x^{\alpha}$ is homogeneous, where $\oN_{=d}^n=\left\{\alpha\in\oN^n:\sum_{i=1}^n\alpha_i=d\right\}.$

Set $s = \frac{1}{rx^*_{\min}}$, so that  $ds \le \epsilon$.
Assume first that the polynomial $f$ has nonnegative coefficients. Then,
using Lemmas \ref{lem:non-neg-hom}, \ref{lem:epsbound} and \ref{lem:sbound}, we can conclude that
\begin{equation}\label{eqtemp1}
f_{\min,\Delta(n,r)} - f_{\min,\Delta_n} \le ((1+s)^d - (1+ds))f_{\min,\Delta_n} \le (1+\epsilon){d \choose 2}s^2f_{\min,\Delta_n}=\frac{(1+\epsilon){d \choose 2}}{r^2(x^*_{\min})^2}f_{\min,\Delta_n}.
\end{equation}
In the general case when no sign condition is assumed on the coefficients of $f$,
we get back to the preceding case by doing a suitable `shift' on $f$.
For this, define the parameters
$$\hat f_{\min}:= \min_{\alpha\in \oN_{=d}^n} f_\alpha {\alpha!\over d!}, \ \ \hat f_{\max}=\max_{\alpha\in \oN_{=d}^n} f_\alpha {\alpha!\over d!}$$
known, respectively, as  the minimum and maximum Bernstein coeffcients of $f$. Observe that, for any $x\in \Delta_n$, $\sum_{\alpha\in \oN_{=d}^n} {d!\over \alpha!}x^\alpha =1$, and thus
$f(x)=\sum_{\alpha\in \oN_{=d}^n} f_\alpha {\alpha!\over d!} \left({d!\over \alpha!}x^\alpha\right)$ is a convex combination of  the Bernstein coefficients $f_\alpha \alpha!/d!$, which implies
\begin{equation}\label{eqBer}
\hat f_{\min}\le f_{\min,\Delta_n}\le f_{\max,\Delta_n}\le \hat f_{\max}.
\end{equation}
 We now define the polynomial
$$g(x)=f(x)-\hat f_{\min}\left(\sum_{i=1}^nx_i\right)^d = \sum_{\alpha\in \oN_{=d}^n} \left(f_\alpha -\hat f_{\min} {d!\over \alpha!}\right) x^\alpha,$$ which is homogeneous of degree $d$ and with nonnegative coefficients.
Hence we can apply the above relation (\ref{eqtemp1})  to $g$ and, since $g$ and $f$ have the same global minimizers on $\Delta_n$, we deduce that
\begin{equation}\label{eqtem2}
f_{\min,\Delta(n,r)} - f_{\min,\Delta_n}=  g_{\min,\Delta(n,r)} - g_{\min,\Delta_n}\le  \frac{(1+\epsilon){d \choose 2}}{r^2(x^*_{\min})^2}g_{\min,\Delta_n}
=  \frac{(1+\epsilon){d \choose 2}}{r^2(x^*_{\min})^2}(f_{\min,\Delta_n}- \hat f_{\min}).
\end{equation}
In view of (\ref{eqBer}), we have: $f_{\min,\Delta_n}- \hat f_{\min} \le \hat f_{\max}-\hat f_{\min}$. Finally, combining with    the inequality:
 $\hat f_{\max}-\hat f_{\min}\le {2d-1\choose d} d^d (f_{\max,\Delta_n}-f_{\min,\Delta_n})$ shown in \cite[Theorem 2.2]{KLP06}, we can conclude the proof.
 \qed\end{proof}

\medskip
Note that Theorem \ref{thm2} does not imply  Theorem \ref{thmklsgene}. Indeed,
if there is a rational global minimizer $x^* \in \Delta(n,m)$, then $x^*_{\min} \ge 1/m$ so that
Theorem \ref{thm2} gives a $O(m^2/r^2)$ bound in terms of $m$ and $r$, as opposed to the
$O(m/r^2)$ bound in  Theorem \ref{thmklsgene}.

\section{Comparison of bounds}

\noindent
We now consider the following seven polynomials, for which we  compare the upper bounds for $f_{\min,\Delta(n,r)}-f_{\min,\Delta_n}$ obtained in Theorem \ref{thmklsquad} or \ref{thmklsgene} (depending on the degree of $f$), Theorem \ref{corst}, \ignore{Corollary \ref{corI},} Corollary \ref{corlast}, and Theorem \ref{thm2}:
{\small
\begin{eqnarray*}
f_1&=&\sum_{i=1}^n\left(x_i-{1\over n}\right)^2,\ \
f_2=\sum_{i=1}^n x_i^2,\ \
f_3=- \sum_{i=1}^n x_i^2,\ \
f_4= \left(x_1-{1\over m}\right)^2+\left(x_2-{m-1\over m}\right)^2,\\
f_5&=& \left(x_1-{m-1\over 2m}\right)^2+\left(x_2-{m+1\over 2m}\right)^2,\ \
f_6=\sum_{i=1}^n x_i^d,\ \ f_{7}=-\prod_{i=1}^{d}x_{i}.
\end{eqnarray*}}

In the first three examples we restrict our attention to the cases when $n \ge 2$. In polynomial $f_4$, we select $m\ge 2$, and in polynomial $f_5$, we select $m$ even, which implies $m-1,m+1$ and $2m$ are relatively prime. In polynomials $f_6$ and $f_7$, we select $d\ge 3$.
The results are shown in Table \ref{tablecom}.

We now summarize the possible relationships between the various bounds in Table \ref{tabledomi}, which should be understood as follows. For instance, having
the entry $f_3$ at the position (Thm. \ref{thmklsquad} or \ref{thmklsgene}, Thm. \ref{corst}) means that, for the polynomial $f_3$, the bound of Theorem \ref{corst} is better than the bound of Theorem \ref{thmklsquad} or \ref{thmklsgene} and this is a strong dominance (since the improvement depends on the parameter $n$). When the improvement depends only on a constant we indicate this by marking the polynomial with an asterix, as for instance for the entry $f_1^{(*)}$ at the position (Thm. \ref{corst}, Thm. \ref{thmklsquad} or \ref{thmklsgene}).
In conclusion, we can see using the polynomials $f_1,f_2,f_3,f_4,f_5$ that there is no possible ordering of the bounds provided by Theorems \ref{thmklsquad} or \ref{thmklsgene}, \ref{corst} and \ref{thm2}  and Corollary
\ignore{\ref{corI}} \ref{corlast}.

\small{
\begin{table}[h!]

\caption{Comparison of  upper bounds\label{tablecom}}
\def\arraystretch{1.5}%
  \begin{tabular}{| c | c  |c|c|c|}\hline
Poly. & Thm. \ref{thmklsquad} or \ref{thmklsgene}    & Thm. \ref{corst} & Cor. \ref{corlast} & Thm. \ref{thm2}\\ \hline
$f_1$ & ${n-1\over r^2}$   & ${(n-1)^2(n+1)^2\over n^3r^2}$ & {${(n-1)^3(n+1)^2 \over n^3r^2}$}& $12(1+\epsilon)(n^2-n)\over r^2$\\
           & $= O({n\over r^2})$   & $= O({n\over r^2})$ &  $=O({n^2\over r^2})$   &    $=O({n^2\over r^2})$
\\ \hline
$f_2$ & ${n-1\over r^2}$   & ${(n-1)^2\over nr^2}$ &  ${(n-1)^2(n^2+n-1)\over n^2r^2}$& {$12(1+\epsilon)(n^2-n)\over r^2$}\\
      & $= O({n\over r^2})$   & $= O({n\over r^2})$ &  $=O({n^2\over r^2})$   &    $=O({n^2\over r^2})$
\\ \hline
$f_3$ & ${n-1\over n r^2}$   &$0$ &$0$& $12(1+\epsilon)(n-1)\over nr^2$\\
      & $= O({1\over r^2})$   &     &   &    $=O({1\over r^2})$
\\ \hline
$f_4$ &${2(m-1)^2\over mr^2}$     & ${5m^2-2m+2\over {4}m^2 r^2}$ &${{(m+1)^2}\over 2m^2 r^2}$ & $24(1+\epsilon)(m-1)^2\over r^2$ \\
      & $= O({m\over r^2})$   & $= O({1\over r^2})$ &  $=O({1\over r^2})$   &    $=O({m^2\over r^2})$
\\ \hline
$f_5$ & ${(m+1)^2\over mr^2}$   & ${9m^2+1\over 8m^2r^2}$ & ${9m^2-6m+1\over {8}m^2r^2}$ & {${24(1+\epsilon)(m+1)^2\over (m-1)^2 r^2}$}\\
      & $= O({m\over r^2})$   & $= O({1\over r^2})$ &  $=O({1\over r^2})$   &    $=O({1\over r^2})$
\\ \hline
$f_6$ & ${d^{2d-1}(1-{1\over d!})(n^d-n)(2d-1)!\over n^{d-1}r^2(d-2)!}$ & ${d(d-1)n\over 2r^2}\left( 1-{1\over n} \right)^2$ & ${(d^2-d)(n^d+(-1)^d(n-1))\over 2r^2n^2/(n-1)^2}$ & ${d^d(1+\epsilon)(n^{d-1}-1)\over n^{d-3}r^2}{d\choose 2}{2d-1\choose d}$ \\
      &  $=O({nd^{3d}\over r^2})$              & $=O({d^2n\over r^2})$ & $=O({d^2n^d\over r^2})$ & $=O({d^{2d+2}n^2\over r^2})$
\\ \hline
$f_7$ & ${d^d(d-1)(d!-1) \over r^2}{2d-1\choose d}$ & ${d(d-1)\over 2r^2}\left(1-{1\over d}\right)^2$ & ${d^2(d-1)\over 2r^2}\left(1-{1\over d}\right)^2$ & ${(1+\epsilon)d^2\over r^2}{d\choose 2}{2d-1\choose d}$ \\
      & $=O({d^{3d+1}\over r^2})$ & $=O({d^2\over r^2})$ & $=O({d^3\over r^2})$ & $=O({d^{d+4}\over r^2})$
\\ \hline
\end{tabular}
\end{table}
}

\begin{table}[h!]
\caption{Possible relationships\label{tabledomi}. If entry $ij$ in the table is $f$, it means that the bound indexed by column $j$ is stronger than the bound indexed by row $i$ for the function $f$.}
\def\arraystretch{1.5}
  \begin{tabular}{| c | c | c |c|c|c|}\hline
 & Thm. \ref{thmklsquad} or \ref{thmklsgene}   & Thm. \ref{corst} & Cor. \ref{corlast} & Thm. \ref{thm2}\\ \hline
Thm. \ref{thmklsquad} or \ref{thmklsgene} & $-$   & $f_2^{(*)},f_3,f_4,f_5,f_6,f_7$ & $f_3,f_4,f_5,f_7$ & $f_5,f_7$ \\ \hline
Thm. \ref{corst} & $f_1 ^{(*)}$    & $-$ & $f_4^{(*)},f_5 ^{(*)}$ & $-$ \\ \hline
Cor. \ref{corlast} & $f_1,f_2$  & $f_1,f_2,f_6,f_7$ & $-$ & $-$ \\ \hline
Thm. \ref{thm2} & $f_1,f_2,f_3^{(*)},f_4$   & $f_1,f_2,f_3,f_4,f_5^{(*)},f_6,f_7$ & $f_1^{(*)},f_2^{(*)},f_3,f_4,f_5^{(*)},f_7$ & $-$\\ \hline
\end{tabular}
\end{table}

\begin{acknowledgements}
We thank the associate editor and two anonymous referees for their comments which helped improve the presentation of the paper.
\end{acknowledgements}


\begin{thebibliography}{12}
\bibitem{AAP}
G.~Ausiello, A.~D'Atri, and M.~Protasi:
Structure preserving reductions among convex optimization problems,
{\em J. Comput. Syst. Sci.} 21(1), 136--153 (1980)

\bibitem{BR}
Bellare, M., Rogaway, P.:
The complexity of approximating a nonlinear program.
\emph{Math. Program.} 69(1), 429--441 (1995)

\bibitem{BK02}
Bomze, I.M., De Klerk, E.: Solving standard quadratic optimization problems via semidefinite and copositive programming. {\em J. Global Optim.} 24(2), 163--185 (2002)


\bibitem{BGY}
Bomze, I.M., Gollowitzer, S., Yildirim, E.A.:
Rounding on the standard simplex: Regular grids for global optimization.
{\em J. Global Optim.} 59(2--3), 243--258 (2014)

\bibitem{Bos83}
Bos, L.P.: Bounding the Lebesque function for Lagrange interpolation in a simplex.
{\em J. Approx. Theory}, 38, 43--59 (1983)


\bibitem{BV04}
Boyd S., Vandenberghe, L.:
{\em Convex Optimization},
Cambridge University Press (2004)

\bibitem{KLP06}
De Klerk, E., Laurent, M., Parrilo, P.:
A PTAS for the minimization of polynomials
of fixed degree over the simplex. {\em Theoret. Comput. Sci.} 361(2--3), 210--225
(2006)

\bibitem{KLS15}
De Klerk, E., Laurent, M., Sun, Z.:
An alternative proof of a PTAS for fixed-degree polynomial optimization over the simplex.
{\em Math. Program.} 151(2), 433--457 (2015)

\bibitem{KLSSIAM}
De Klerk, E., Laurent, M., Sun, Z.:
An error analysis for polynomial optimization over the simplex based on the multivariate hypergeometric distribution.
{\em SIAM J. Optim.} 25(3), 1498--1514 (2015)

\bibitem{Fay03}
L.~Faybusovich:
Global optimization of homogeneous polynomials on the simplex and on
  the sphere,
 In C.~Floudas and P.~Pardalos (eds.), {\em Frontiers in Global
  Optimization}, (Kluwer Academic Publishers, 2004), 109--121.

  \bibitem{Nes2003}
Yu.~Nesterov:
Random walk in a simplex and quadratic optimization over convex polytopes.
\emph{CORE Discussion  Paper 2003/71}, CORE-UCL (2003).

  \end{thebibliography}
\end{document}